\documentclass[12pt,reqno]{amsart}

\usepackage{amssymb,amsmath,graphicx,amsfonts,euscript}
\usepackage{color}
\usepackage{cite}

\allowdisplaybreaks

\newcommand{\om}{\omega}

\newcommand{\na}{\nabla}
\def\R{\mathbb{R}}
\newcommand{\p}{\partial}

\newtheorem{thm}{Theorem}[section]

\newtheorem{lem}{Lemma}[section]

\newcommand{\beq}{\begin{equation}}
\newcommand{\eeq}{\end{equation}}
\newcommand{\ben}{\begin{eqnarray}}
\newcommand{\een}{\end{eqnarray}}
\newcommand{\beno}{\begin{eqnarray*}}
\newcommand{\eeno}{\end{eqnarray*}}

\makeatletter
\@namedef{subjclassname@2020}{%
	\textup{2020} Mathematics Subject Classification}
\makeatother

\numberwithin{equation}{section}

\begin{document}
	
	\title[Stability and exponential decay for Boussinesq equations]
	{Stability and exponential decay for the 2D anisotropic Boussinesq equations with horizontal dissipation}
	
 \author[Dong, Wu, Xu and Zhu]{Boqing Dong$^{1}$, Jiahong Wu$^{2}$, Xiaojing Xu$^{3}$ and Ning Zhu$^{4}$} 
	
\address{$^1$ College of Mathematics and Statistics, Shenzhen University, Shenzhen 518060, P.R. China}
	
	\email{bqdong@szu.edu.cn}
	
\address{$^2$ Department of Mathematics, Oklahoma State University, Stillwater, OK 74078, United States}
	
	\email{jiahong.wu@okstate.edu}
	
\address{$^3$ Laboratory of Mathematics and Complex Systems, Ministry of Education, School of Mathematical Sciences, Beijing Normal University, Beijing 100875, P.R. China}	

		\email{xjxu@bnu.edu.cn}

\address{$^4$ School of Mathematical Sciences, Peking University, Beijing 100871, P.R. China}
		
		\email{mathzhu1@163.com}

	\vskip .2in
	\begin{abstract}
The hydrostatic equilibrium is a prominent topic in fluid dynamics and astrophysics. Understanding the stability of perturbations near the hydrostatic equilibrium of the Boussinesq {{system}} helps gain insight into certain weather phenomena. The 2D Boussinesq system focused here is anisotropic and involves only horizontal dissipation
and horizontal thermal diffusion. Due to the lack of the vertical dissipation, the stability and precise large-time behavior problem is difficult. When the spatial domain is $\mathbb R^2$, the stability problem in a Sobolev setting
remains open. When the spatial domain is $\mathbb T\times \mathbb R$, this paper solves 
the stability problem and specifies the precise large-time behavior of the perturbation. By decomposing the velocity $u$ and 
temperature $\theta$ into the horizontal average $(\bar u, \bar\theta)$  and the corresponding oscillation $(\widetilde u, \widetilde \theta)$,  and deriving various anisotropic inequalities, we are able to establish the global stability in the Sobolev space $H^2$. In addition, we prove that the oscillation $(\widetilde u, \widetilde \theta)$ decays exponentially to zero in $H^1$ and $(u, \theta)$ converges to $(\bar u, \bar\theta)$. This result reflects the stratification phenomenon of buoyancy-driven fluids.
\end{abstract}
	\maketitle

\vskip .2in 
\section{Introduction}

{{The}} goal of this paper is to understand the stability and large-time behavior problem on perturbations 
near the hydrostatic equilibrium of buoyancy-driven fluids. Being capable of capturing the key features of 
buoyancy-driven fluids such as stratification, the Boussinesq equations have become the most frequently used
models for these circumstances (see, e.s., \cite{Maj, Pe}). The Boussinesq system concerned here assumes the form 
\beq\label{boussiensq}
\begin{cases}
	\p_t u + u\cdot\na u = -\na P + \nu \p_{11} u+\Theta e_2,  \\
	\p_t \Theta + u\cdot\na \Theta =  \kappa \p_{11} \Theta,  \\
	\na\cdot u=0,
\end{cases}
\eeq
where $u = (u_1,u_2)$ denotes the velocity field, $P$ the pressure, $\Theta$ the temperature, $e_2=(0,1)$ (the unit vector in the vertical direction), and $\nu > 0$ and $\kappa > 0$ are the viscosity and the thermal diffusivity, respectively. (\ref{boussiensq}) involves only horizontal 
dissipation and horizontal thermal diffusion, and governs the motion of anisotropic fluids when the corresponding 
vertical dissipation and thermal diffusion are negligible (see, e.g., \cite{Pe}).

\vskip .1in 
The Boussinesq systems have attracted considerable interests recently due
to their broad physical applications and mathematical significance. The Boussinesq systems are the most frequently 
used models for buoyancy-driven fluids such as many large-scale geophysical flows and the 
{{Rayleigh-B\'{e}nard}} convection (see, e.g., \cite{ConD, DoeringG,Maj,Pe}). The Boussinesq equations 
are also mathematically important. They share many similarities with the 3D Navier-Stokes and 
the Euler equations. In fact, the 2D Boussinesq equations retain some key features of the 3D Euler and Navier-Stokes equations such as the vortex stretching mechanism.  The inviscid 2D Boussinesq equations can be identified as the Euler equations for the 3D axisymmetric swirling flows \cite{MaBe}.

\vskip .1in 
Many efforts have been devoted to understanding two fundamental problems concerning the Boussinesq systems. 
The first is the global existence and regularity problem. Substantial progress has been made on various Boussinesq
systems, especially those with only partial or fractional dissipation (see, e.g., \cite{ACSW,ACW10,ACW11,ACWX,CaoWu1,Ch,ChConWu,Chae-Nam1997,ChaeWu,CKY,ConVWu,
	DP2,DP3,ElJ,Tarek2,LBHe,HKR1,HKR2,
	HL,Hu,HuWangWu,JMWZ,JWYang,KRTW,KiTan,LaiPan,LLT,LSWu,MX,PaicuZhu, SaWu,SteWu,Wen,Wu,Wu_Xu,Wu_Xu_Xue_Ye,Wu_Xu_Zhu,
	Wu_Xu_Ye,Xu0,YangJW,YangJW2,YeXu,zhao2010}).
The second is the stability problem on perturbations near several physically relevant steady states. 
The investigations on the stability problem {{are}} relatively more recent. One particular important  
steady state is the hydrostatic equilibrium. Mathematically the hydrostatic equilibrium refers to the 
stationary solution $(u_{he}, \Theta_{he}, P_{he})$ with 
$$
u_{he} =0, \quad \Theta_{he} =x_2, \quad P_{he}= \frac12 x_2^2.
$$
The hydrostatic equilibrium is 
one of the most prominent topics in fluid dynamics, atmospherics and astrophysics. In fact, our atmosphere 
is mostly in the hydrostatic equilibrium with the upward pressure-gradient force balanced out by 
the downward gravity. The work of Doering, Wu, Zhao and Zheng \cite{DWZZ} initiated the rigorous study on the stability
problem near the hydrostatic equilibrium of the 2D Boussinesq equations with only velocity dissipation. A 
followup work of Tao, Wu, Zhao and Zheng establishes the large-time behavior and the eventual temperature profile \cite{TWZZ}. The paper of Castro,  C\'ordoba and Lear successfully  established the stability and large time behavior on 
the 2D Boussinesq equations with velocity damping instead of dissipation \cite{CCL}. There are other more recent work (\cite{Ben, Wan, Wu_Xu_Zhu,Zil}). Another important steady state is the shear flow. Linear stability results on the shear flow for several 
partially dissipated Boussinesq systems are obtained in \cite{TW} and \cite{Zil} while the nonlinear stability problem 
on the shear flow of the 2D Boussinesq equations with only vertical dissipation was solved by \cite{DWZ}. 

\vskip .1in 
The goal of this paper is to assess the stability and the precise large-time 
behavior of perturbations near the hydrostatic equilibrium. To understand the stability problem, we write the equations for the perturbation
$(u, p, \theta)$ where 
$$
p= P- \frac12 x_2^2 \quad \mbox{and}\quad \theta= \Theta-x_2.
$$
It is easy to check that $(u,p,\theta)$ satisfies
\beq\label{B1}
\begin{cases}
	\p_t u + u\cdot\na u = -\na p + \nu \,\p_{11} u+\theta e_2,\\
	\p_t \theta + u\cdot\na \theta+u_2 =  \kappa\, \p_{11} \theta,  \\
	\na\cdot u=0, \\
	u(x,0) =u_0(x), \quad\theta(x,0)=\theta_0(x).
\end{cases}
\eeq
We have observed a new phenomenon on (\ref{B1}). It appears that the type of the spatial domain plays a crucial role 
in the resolution of the stability problem concerned here. 

\vskip .1in 
When the spatial domain is the whole plane $\mathbb R^2$, the stability problem remains an open 
problem. Given any sufficiently smooth initial data $(u_0, \theta_0)\in H^2(\mathbb R^2)$,  (\ref{B1}) does 
admit a unique global solution. But the solution could potentially grow rather rapidly in time. In fact,
the best upper bounds one could obtain on $\|u(t)\|_{H^1}$ and $\|\theta(t)\|_{H^1}$ grow algebraically in time. The only way to possibly establish upper bounds that are uniform in time is to combine the equation of $\na u$ and that of $\na \theta$, or equivalently the equation of $\om$ with $\na \theta$, where $\om=\na \times u$ is the 
vorticity. When we combine the equations of $\om$ and $\na \theta$, 
\beq\label{gh}
\begin{cases} 
\p_t \om + u\cdot\na \om  = \nu \, \p_{11} \om + \p_1 \theta, \\
\p_t \na \theta + u\cdot\na (\na \theta) + \na u_2 = \kappa\, \na \p_{11} \theta - \na u\cdot \na \theta
\end{cases} 
\eeq
and estimate $\|\om\|_{L^2}$ and $\|\na \theta\|_{L^2}$ simultaneously, we can eliminate the term from the 
buoyancy, namely $\p_1 \theta$. In fact, a simple energy estimate on (\ref{gh}) yields, after suitable integration by parts, 
\begin{align}
&\frac{d}{dt} \left(\|\om(t)\|_{L^2}^2 + \|\na \theta(t)\|_{L^2}^2\right) + 2 \nu \,\|\p_1 \om(t)\|_{L^2}^2 
+ 2 \kappa \,\|\p_1 \na \theta(t)\|_{L^2}^2 \notag \\
&= -\int_{\mathbb R^2} \na \theta \cdot \na u \cdot \na \theta\,dx.\label{ener}
\end{align}
The difficulty is how to obtain a suitable upper bound on the term on the right-hand side of (\ref{ener}). 
To make full use of the anisotropic dissipation, we naturally divide this term further into four component terms
\ben
- \int_{\mathbb R^2} \na \theta \cdot \na u \cdot \na \theta\,dx &=& - \int_{\mathbb R^2} \p_1 u_1 (\p_1 \theta)^2\,dx  - \int_{\mathbb R^2} \p_1 u_2 \, \p_1\theta\,\p_2\theta\,dx \notag\\
&& - \int_{\mathbb R^2} \p_2 u_1 \, \p_1\theta\,\p_2\theta\,dx - \int_{\mathbb R^2} \p_2 u_2 (\p_2 \theta)^2\,dx.
\label{hh}
\een
Due to the lack of dissipation or thermal diffusion in the vertical direction, the {last} two terms 
on (\ref{hh}) prevents us from bounding them suitably. This is one of the difficulties that keep the stability 
problem on (\ref{B1}) open when the spatial domain is the whole plane $\mathbb R^2$. 

\vskip .1in 
When the spatial domain is 
$$
\Omega = \mathbb T\times \mathbb R
$$
with $\mathbb T=[0, 1]$ being a 1D periodic box and $\mathbb R$ being the whole line, this paper is able to 
solve the desired stability problem on (\ref{B1}). In fact, we are able to prove the following result. 
\begin{thm}\label{main1}
	Let $\mathbb T=[0, 1]$ be a 1D periodic box and let $\Omega= \mathbb T\times \mathbb R$.  
	Assume $u_0,\theta_0\in H^2(\Omega)$ and $\na\cdot u_0=0$. Then there exists $\varepsilon>0$ such that, if
\begin{equation}
\label{smallness condition}
\|u_0\|_{H^2}+\|\theta_0\|_{H^2}\leq \varepsilon,
\end{equation}
	then (\ref{B1}) has a unique global solution that remains uniformly bounded for all time,
	$$
	\|u(t)\|_{H^2}^2+\|\theta(t)\|_{H^2}^2+ \nu \int_0^t \|\p_1 u(\tau)\|_{H^2}^2 \,d\tau+\kappa \int_0^t \|\p_1 \theta(\tau)\|_{H^2}^2 \,d\tau \le \, C^2_0 \varepsilon^2
	$$
	for some pure constant $C_0>0$ and for all $t>0$. 
\end{thm}

How does this domain makes a difference? The key point is that $\Omega$ allows us to separate 
the horizontal average (or the zeroth horizontal Fourier mode) from the corresponding oscillation 
part. These two different parts have different physical behavior. In fact, this decomposition is partially
motivated by the stratification phenomenon observed in numerical results \cite{DWZZ}.  The numerical
simulations performed in \cite{DWZZ} show that the  temperature {{becomes}} horizontally homogeneous and stratify in the vertical direction as time evolves. Mathematically we do not expect 
the horizontal average to decay in time since it is associated with the zeroth horizontal Fourier 
mode and the dissipative effect at this mode vanishes. The oscillation part could decay exponentially. 
In addition, this decomposition and  the oscillation part possess several desirable mathematical 
properties such as a strong Poincar\'{e} type inequality. The two difficult terms in (\ref{hh}) are now 
handled by decomposing both $u$ and $\theta$ into the aforementioned two parts, and different terms induced
by the decomposition are estimated differently. This is the main reason why the impossible stability problem in the $\mathbb R^2$
case becomes solvable when the domain is $\Omega = \mathbb T\times \mathbb R$. 

\vskip .1in 
To make the idea described above more precise, we introduce a few notations. Since the functional setting 
for our solution $(u, \theta)$ is $H^2(\Omega)$, it is meaningful to define the horizontal average, 
$$
\bar{u}(x_2, t) = \int_{\mathbb T} u(x_1, x_2, t)\,dx_1.
$$
We set $\widetilde u$ to be the corresponding oscillation part  
$$
\widetilde u = u - \bar{u} \quad \mbox{or}\quad u = \bar{u} + \widetilde u.
$$
$\bar{\theta}$ and $\widetilde \theta$ are similarly defined. This decomposition is orthogonal, 
$$
(\bar u, \widetilde u) := \int_{\Omega} \bar{u}\, \widetilde u\,dx =0, \qquad 
\|u\|_{L^2(\Omega)}^2 = \|\bar{u}\|_{L^2(\Omega)}^2 + \|\widetilde u\|_{L^2(\Omega)}^2. 
$$ 
A crucial property of $\widetilde u$ is that it obeys a strong version of the Poincar\'{e} type 
inequality,
\beq\label{pp}
\|\widetilde u\|_{L^2(\Omega)} \le C\, \|\p_1 \widetilde u\|_{L^2(\Omega)},
\eeq
where the full gradient in the standard Poincar\'{e} type inequality is replaced by $\p_1$. 
With this decomposition at our disposal, we are ready to handle the two difficult terms in (\ref{hh}). 
For the sake of conciseness, we focus on the first term. By invoking the decomposition, we can further 
split it into four terms, 
\begin{align}
\int_\Omega \p_2 u_1 \, \p_1 \om\, \p_2 \om\,dx &= \int_\Omega \p_2(\bar u_1 + \widetilde u_1)
\, \p_1(\bar \om + \widetilde \om)\,\p_2(\bar \om + \widetilde \om)\,dx \notag\\
&= \int_\Omega \p_2 \bar u_1 \, \p_1 \widetilde\om\, \p_2 \bar\om\,dx + \int_\Omega \p_2 \bar u_1 \, \p_1\widetilde \om\, \p_2 \widetilde\om\,dx\notag\\
&\quad  + \int_\Omega \p_2 \widetilde u_1 \, \p_1 \widetilde \om\, \p_2 \bar \om\,dx + \int_\Omega \p_2 \widetilde u_1 \, \p_1 \widetilde \om\, \p_2 \widetilde \om\,dx, \label{dde}
\end{align}
where we have used $\p_1 \bar \om =0$. The first term in (\ref{dde}) is clearly zero, 
$$
\int_\Omega \p_2 \bar u_1 \, \p_1 \widetilde\om\, \p_2 \bar\om\,dx =0.
$$
The other three terms in (\ref{dde}) can all be bounded suitably by applying (\ref{pp}) and several other 
anisotropic inequalities, as stated in the lemmas in Section \ref{secA}. We leave more technical details to the proof of 
Theorem \ref{main1} in Section \ref{sec3}. 

\vskip .1in 
Our second main result states that the oscillation part $(\widetilde u, \widetilde \theta)$ decays 
to zero exponentially in time in the $H^1$-norm. This result reflects the stratification phenomenon 
of buoyancy driven fluids. It also rigorously confirms the observation of the numerical 
simulations in \cite {DWZZ}, the temperature eventually stratifies and converges to the horizontal average.  
As we have explained before, the horizontal average $(\bar u, \bar \theta)$ is not expected to decay in time. 

\begin{thm}\label{main2}  
	Let $u_0,\theta_0\in H^2(\Omega)$ with $\na\cdot u_0=0$. Assume that $(u_0, \theta_0)$ satisfies \eqref{smallness condition} for sufficiently small $\varepsilon>0$. Let $(u, \theta)$ be the corresponding solution of (\ref{B1}). Then the $H^1$ norm of the oscillation part $(\widetilde{u},\widetilde{\theta})$ decays exponentially in time, 
	$$\|\widetilde{u}(t)\|_{H^1}+\|\widetilde{\theta}(t)\|_{H^1}\leq  (\|{u}_0\|_{H^1}+\|{\theta}_0\|_{H^1})e^{-C_1t}, $$
	for some pure constant $C_1>0$ and for all $t>0$. 
\end{thm}

As a special consequence of this decay result, the solution $(u, \theta)$ approaches the horizontal average 
$(\bar u, \bar \theta)$ asymptotically, and the Boussinesq system  (\ref{B1}) evolves to the following 1D system 
$$
\begin{cases}
	\partial_{t} \bar{u}+\overline{u\cdot \nabla \widetilde{u}}+\left(\begin{array}{c}
		{0} \\
		{\partial_{2} \bar{p}}
	\end{array}\right)
	=\left(\begin{array}{c}
		{0} \\
		{\bar{\theta}}
	\end{array}\right), \\
	\partial_t \bar{\theta}+\overline{u\cdot \nabla \widetilde{\theta}}=0,
\end{cases}
$$
as given in (\ref{bar-boussinesq}). The proof of Theorem \ref{main2} starts with the system 
governing the oscillation $(\widetilde u, \widetilde \theta)$,
$$
\begin{cases}
	\p_{t} \widetilde{u}+\widetilde{u\cdot \nabla \widetilde{u}}+u_2\p_2\bar{u}-\nu\p_1^2\widetilde{u}+\nabla\widetilde{p}
	=\widetilde{\theta}e_2,\\
	\p_{t} \widetilde{\theta}+\widetilde{u\cdot \nabla \widetilde{\theta}}+u_2\p_2\bar{\theta}-{\kappa\p_1^2\widetilde{\theta}}+\widetilde{u}_2
	=0.
\end{cases}
$$
By performing separate energy estimates for $\|(\widetilde u, \widetilde \theta)\|_{L^2}^2$ 
and $\|(\na \widetilde u, \na\widetilde \theta)\|_{L^2}^2$ and carefully evaluating 
the nonlinear terms with the strong Poincar\'{e}
type inequality and other anisotropic tools, we are able to establish the inequality 
\ben\label{pq}
\frac{d}{dt} \|(\widetilde u, \widetilde \theta)\|_{H^1}^2 + (2\nu -C_1\,\|(u, \theta)\|_{H^2}) \,\|\p_1\widetilde u\|_{H^1}^2 \notag \\
 + (2\kappa -C_1\,\|(u, \theta)\|_{H^2}) \,\|\p_1\widetilde \theta\|_{H^1}^2 \le 0.
\een
When the initial data $(u_0, \theta_0)$ is taken to be sufficiently small in $H^2$, say 
$$
\|(u_0, \theta_0)\|_{H^2} \le \varepsilon
$$
for sufficiently small $\varepsilon>0$, then $\|(u, \theta)\|_{H^2} \le C_0\, \varepsilon$ and
$$
2\nu -C_1\,\|(u, \theta)\|_{H^2} \ge \nu, \qquad 2\eta -C_1\,\|(u, \theta)\|_{H^2} \ge \eta.
$$
Applying the strong Poincar\'{e} inequality to (\ref{pq}) yields the desired exponential decay. 

\vskip .1in 
The rest of this paper is divided into three sections. Section \ref{secA} serves as a preparation. It presents several anisotropic inequalities and some fine properties related to the orthogonal decomposition. Section 
\ref{sec3} proves Theorem \ref{main1} while Section \ref{sec4} is devoted to verifying Theorem \ref{main2}.

\vskip .3in 
\section{Anisotropic inequalities}
\label{secA}

This section presents several anisotropic inequalities to be used extensively in the proofs 
of Theorem \ref{main1} and  Theorem \ref{main2}. In addition, several key properties of the horizontal average 
and the corresponding oscillation are also listed for the convenience of later applications.  
	
\vskip .1in 
We first recall the horizontal average and the corresponding orthogonal decomposition. For any function 
$f=f(x_1, x_2)$ that is integrable in $x_1$ over the 1D periodic box  $\mathbb T  =[0, 1]$, its horizontal 
average $\bar f$ is given by 
\beq\label{fbar}
	\bar f (x_2) = \int_{\mathbb T} f(x_1, x_2) \,dx_1. 
\eeq 	
We decompose $f$ into $\bar f$ and the corresponding oscillation portion $\widetilde f$, 
\beq\label{de}
 f = \bar f + \widetilde f. 
\eeq 	
The following lemma collects a few properties of $\bar f$ and $\widetilde f$ to be used in the subsequent 
sections. These properties can be easily verified via (\ref{fbar}) and (\ref{de}). 	
	
\begin{lem}\label{goo1}
	Assume that the 2D function $f$ defined on $\Omega=\mathbb T\times \mathbb R$ is sufficiently regular, say 
	$f\in H^2(\Omega)$. Let $\bar f$ and $\widetilde f$ be defined as in (\ref{fbar}) and (\ref{de}). 
	\begin{enumerate}
		\item[(a)] $\bar f$ and $\widetilde f$ obey the following basic properties, 
			$$
		\overline{\p_1 f} = \p_1 \bar f = 0, \quad  \overline{\p_2 f}  = \p_2 \bar f, \quad\overline{\widetilde f} =0, \quad  \widetilde{\p_2 f}  = \p_2 \widetilde f.
		$$ 
		\item[(b)] If $f$ is a divergence-free vector field, namely $\na\cdot f=0$, then $\bar f$ and 
		$\widetilde f$ are also divergence-free, 
		$$
		\na\cdot \bar f =0\quad\mbox{and}\quad \na\cdot \widetilde f=0.
		$$
		\item[(c)] $\bar f$ and $\widetilde f$ are orthogonal in $L^2$, namely 
		$$
		(\bar f, \widetilde f) := \int _{\Omega} \bar f\, \widetilde f\,dx = 0, \quad \|f\|_{L^2(\Omega)}^2 = \|\bar f\|_{L^2(\Omega)}^2 + \|\widetilde f\|_{L^2(\Omega)}^2.
		$$
		In particular, $\|\bar f\|_{L^2} \le \|f\|_{L^2}$ and $\|\widetilde f\|_{L^2} \le \|f\|_{L^2}$.
	\end{enumerate}
\end{lem}

\begin{proof}
(a) follows from the definitions of $\bar{f}$ and $\widetilde{f}$ directly. If $\nabla \cdot f=0$, then
\begin{equation*}
0=\overline{\partial_1f+\p_2f}=\p_1\bar{f}+\p_2\bar{f}=\nabla\cdot\bar{f}=\nabla\cdot f-\nabla\cdot\widetilde{f}=-\nabla\cdot\widetilde{f},
\end{equation*}
which gives (b). For (c), according to the definitions of $\bar{f}$ and $\widetilde{f}$,
\begin{equation*}
(\bar{f},\widetilde{f})=\int_\Omega\bar{f}\widetilde{f}~dx=\int_\R\bar{f}\bigg(\int_{\mathbb{T}}f(x_1,x_2)~dx_1\bigg)~dx_2-\int_\R|\bar{f}|^2~dx_2=0.
\end{equation*}
This completes the proof of Lemma \ref{goo1}.
\end{proof}
	
\vskip .1in 
We now present several anisotropic inequalities. Basic 1D inequalities play a role 
in these anisotropic  inequalities. We emphasize that 1D inequalities on the whole line $\mathbb R$ 
are not always the same as the corresponding ones on bounded domains including periodic domains. For any 
1D function $f \in H^1(\mathbb R)$, 
\beq\label{gg1}
\|f\|_{L^\infty(\mathbb R)} \le \sqrt{2} \|f\|_{L^2(\mathbb R)}^{\frac12} \,  \|f'\|_{L^2(\mathbb R)}^{\frac12}. 
\eeq
For a bounded domain such as $\mathbb T$ and $f\in H^1(\mathbb T)$, 
\beq\label{gg2}
\|f\|_{L^\infty(\mathbb T)} \le \sqrt{2} \,\|f\|^{\frac12}_{L^2(\mathbb T)} \,\|f'\|^{\frac12}_{L^2(\mathbb T)} + \|f\|_{L^2(\mathbb T)}.
\eeq	
However, if a function has mean zero such as the oscillation part $\widetilde f$, 
the 1D inequality for $\widetilde f$ is the same as the whole line case, that is, for $\widetilde f\in H^1(\mathbb T)$, 
\beq\label{w2}
\|\widetilde f\|_{L^\infty(\mathbb T)} \le C\, \|\widetilde f\|^{\frac12}_{L^2(\mathbb T)}\, 
\| (\widetilde f)' \|^{\frac12}_{L^2(\mathbb T)}.
\eeq
These basic inequalities are incorporated into the anisotropic inequalities stated in the following lemmas. 

\begin{lem} \label{w1}
	Let $\Omega = \mathbb T\times \mathbb R$. For any $f,g,h\in L^2(\Omega)$ with $\p_1 f\in L^2(\Omega)$ and 
	$\p_2 g \in L^2(\Omega)$, then 
	\ben\label{ani}
	\left|\int_{\Omega} f \, g\, h\,dx \right| \le C\, \|f\|_{L^2}^{\frac12} \left(\|f\|_{L^2} +  \|\p_1f\|_{L^2}\right)^{\frac12} \|g\|_{L^2}^{\frac12}\, \|\p_2 g\|_{L^2}^{\frac12}\, \|h\|_{L^2}.
	\een
	For any $f\in H^2(\Omega)$, we have 
	\ben\label{2gg}
	\|f\|_{L^\infty(\Omega)} &\le& C\, \|f\|^{\frac14}_{L^2(\Omega)}\left(\|f\|_{L^2(\Omega)} + \|\p_1f\|_{L^2(\Omega)}\right)^{\frac14} \|\p_2f\|^{\frac14}_{L^2(\Omega)}\notag\\
	&& \times \left(\|\p_2f\|_{L^2(\Omega)} + \|\p_1\p_2f\|_{L^2(\Omega)}\right)^{\frac14}.
	\een
\end{lem}

\begin{proof}	
The upper bound for the triple product in (\ref{ani}) on $\mathbb R^2$ was stated and proven 
in \cite{caowu2011}, but (\ref{ani}) for the domain $\Omega$ includes an extra lower-order term. 
For the convenience of the readers, we provide the proofs of (\ref{ani}) and (\ref{2gg}). Applying
H\"{o}lder's inequality in each direction,  Minkowski's inequality, and (\ref{gg1}) and (\ref{gg2}), 
we have 
\begin{align*}
\left|\int_{\Omega} f \, g\, h\,dx \right| &\le \|f\|_{L^2_{x_2}L^\infty_{x_1}}\,\|g\|_{L^\infty_{x_2}L^2_{x_1}}\,\|h\|_{L^2}\\
&\le \|f\|_{L^2_{x_2}L^\infty_{x_1}}\,\|g\|_{L^2_{x_1}L^\infty_{x_2}}\,\|h\|_{L^2}\\
&\le  C\, \left\|\|f\|_{L^2_{x_1}}^{\frac12} \, \|\p_1 f\|_{L^2_{x_1}}^{\frac12} +\|f\|_{L^2_{x_1}}\right\|_{L^2_{x_2}}\\
&\quad \times  \left\|\|g\|_{L^2_{x_2}}^{\frac12} \, \|\p_2 g\|_{L^2_{x_2}}^{\frac12}\right\|_{L^2_{x_1}} \, \|h\|_{L^2}\\
&\le C\, \|f\|_{L^2}^{\frac12} \left(\|f\|_{L^2} +  \|\p_1f\|_{L^2}\right)^{\frac12} \|g\|_{L^2}^{\frac12}\, \|\p_2 g\|_{L^2}^{\frac12}\, \|h\|_{L^2}.
\end{align*}
Here $\|f\|_{L^2_{x_2}L^\infty_{x_1}}$ represents the $L^\infty$-norm in the $x_1$-variable, followed by the $L^2$-norm in the $x_2$-variable.
To prove (\ref{2gg}), we again use H\"{o}lder's inequality,  Minkowski's inequality, and (\ref{gg1}) and (\ref{gg2}),
\begin{align*}
\|f\|_{L^\infty_{x_1} L^\infty_{x_2}} &\le C\,  \left\|\|f\|_{L^2_{x_2}}^{\frac12} \, \|\p_2 f\|_{L^2_{x_2}}^{\frac12}\right\|_{L^\infty_{x_1} } \\
&\le C\, \left\|\|f\|_{L^\infty_{x_1}}\right\|^{\frac12}_{L^2_{x_2}}
\,\left\|\|\p_2 f\|_{L^\infty_{x_1}}\right\|^{\frac12}_{L^2_{x_2}}\\
&\le C\, \left\|\|f\|_{L^2_{x_1}}^{\frac12} \, \|\p_1 f\|_{L^2_{x_1}}^{\frac12} +\|f\|_{L^2_{x_1}} \right\|^{\frac12}_{L^2_{x_2}}\\
&\quad \times \left\|\|\p_2f\|_{L^2_{x_1}}^{\frac12} \, \|\p_1\p_2 f\|_{L^2_{x_1}}^{\frac12} +\|\p_2f\|_{L^2_{x_1}} \right\|^{\frac12}_{L^2_{x_2}} \\
&\le C\,   \|f\|^{\frac14}_{L^2}\left(\|f\|_{L^2} + \|\p_1f\|_{L^2}\right)^{\frac14} \|\p_2f\|^{\frac14}_{L^2}\notag\\
&\quad \times \left(\|\p_2f\|_{L^2} + \|\p_1\p_2f\|_{L^2}\right)^{\frac14}.
\end{align*}
This completes the proof of Lemma \ref{w1}. 
\end{proof}

\vskip .1in 
If we replace $f$ by the oscillation part $\widetilde f$, some of the lower-order parts in (\ref{ani}) and (\ref{2gg}) can be  dropped, as the following lemma states. 

\begin{lem} \label{goo4}
	Let $\Omega = \mathbb T\times \mathbb R$.  For any $f,g,h\in L^2(\Omega)$ with $\p_1 f\in L^2(\Omega)$ and 
	$\p_2 g \in L^2(\Omega)$, then 
	\ben
	\left|\int_{\Omega} \widetilde f \, g\, h\,dx \right| &\le& C\, \|\widetilde f\|_{L^2}^{\frac12}\, 
	\|\p_1\widetilde f\|_{L^2}^{\frac12}\, \|g\|_{L^2}^{\frac12}\, \|\p_2 g\|_{L^2}^{\frac12}\, \|h\|_{L^2}.
	\label{an2}
	\een
	For any $f\in H^2(\Omega)$, we have 
	$$
	\|\widetilde f\|_{L^\infty(\Omega)} \le C\, \|\widetilde f\|^{\frac14}_{L^2(\Omega)} \|\p_1\widetilde f\|_{L^2(\Omega)}^{\frac14} \,\|\p_2 \widetilde f\|^{\frac14}_{L^2(\Omega)}\, \|\p_1\p_2\widetilde f\|^{\frac14}_{L^2(\Omega)}.
	$$
\end{lem}

	\begin{proof}	The two inequalities in this lemma can be shown similarly as those in
		Lemma \ref{w1}. The only modification here is to use (\ref{w2}) instead of (\ref{gg2}). Since
		(\ref{w2}) does not contain the lower-order part, the inequalities in this lemma do not have 
		the lower-order terms. 		
\end{proof}

\vskip .1in 
The next lemma assesses that the oscillation part $\widetilde f$ obeys a strong Poincar\'{e} 
type inequality with the upper bound in terms of $\p_1 \widetilde f$ instead of $\na\widetilde f$. 

\begin{lem} \label{goo2}
	Let {$f$ be a smooth function,} $\bar f$ and $\widetilde f$ be defined as in (\ref{fbar}) and (\ref{de}).  If $\|\p_1\widetilde f\|_{L^2(\Omega)} <\infty$, then 
	\ben\label{poincare1}
	\|\widetilde f\|_{L^2(\Omega)} \le C \|\p_1\widetilde f\|_{L^2(\Omega)}, 
	\een
	where $C$ is a pure constant. In addition, if $\|\p_1\widetilde f\|_{H^1(\Omega)} <\infty$, then 
	\ben\label{poincare2}
	\|\widetilde f\|_{L^\infty(\Omega)} \le C \|\p_1\widetilde f\|_{H^1(\Omega)}.
	\een
\end{lem}

\begin{proof}
For fixed $x_2\in \mathbb R$, 
\beno
\int_0^1 \widetilde{f}(x_1,x_2)~dx_1=0.
\eeno
According to the mean-value theorem, there exists $\eta\in [0,1]$ such that $\widetilde{f}(\eta,x_2)=0$.
Therefore, by H\"older inequality,
\beq\label{w3}
 |\widetilde{f}(x_1,x_2)| = \left|\int^{x_1}_{\eta}\p_{y_1}\widetilde{f}(y_1,x_2)~dy_1\right| \leq \bigg(\int^{1}_{0}(\p_{y_1}\widetilde{f}(y_1,x_2))^2~dy_1\bigg)^{\frac12}.
\eeq
Taking the $L^2$-norm of (\ref{w3}) over $\Omega$ yields
	\beno
\|\widetilde f\|_{L^2(\Omega)} \le C \|\p_1\widetilde f\|_{L^2(\Omega)}. 
\eeno
To prove \eqref{poincare2}, we use \eqref{gg2}, namely, for any fixed $x_2\in \R$, 
	\beno
\|\widetilde f(x_1,x_2)\|_{L^\infty_{x_1}(\mathbb{T})} \le C \|\widetilde f\|_{L^2_{x_1}(\mathbb{T})}^{\frac12} \|\p_1\widetilde f\|_{L^2_{x_1}(\mathbb{T})}^{\frac12}.
\eeno
Taking the $L^\infty$-norm in $x_2$,  and using \eqref{gg1} and \eqref{poincare1}, we find
	\beno
\|\widetilde f\|_{L^\infty(\Omega)} \le C \|\p_1\widetilde f\|_{H^1(\Omega)}.
\eeno
This  completes the proof of Lemma \ref{goo2}.
\end{proof}

As an application of Lemma \ref{goo2}, the inequality in (\ref{an2}) can be converted to 
	$$
	\left|\int_{\Omega} \widetilde f \, g\, h\,dx \right| \le C\, \|\p_1\widetilde f\|_{L^2}\, \|g\|_{L^2}^{\frac12}\, \|\p_2 g\|_{L^2}^{\frac12}\, \|h\|_{L^2}.
$$

\vskip .3in 
\section{Proof of Theorem \ref{main1}}
\label{sec3}

This section is devoted to the proof of Theorem \ref{main1}. Since the local well-posedness can be shown via 
standard methods such as Friedrichs' Fourier cutoff, our focus is on the global {\it a priori} bound on 
the solution in $H^2(\Omega)$. 

\vskip .1in 
The framework for proving the global $H^2$-bound is 
the bootstrapping argument (see, e.g., \cite[p.21]{Tao2006}). By selecting a suitable energy functional 
at the $H^2$-level, we devote our main efforts to showing that this energy functional obeys a desirable 
energy inequality. This process is lengthy and involves establishing suitable upper bounds for several 
nonlinear terms such as the one in (\ref{dde}). As described in the introduction, we invoke the orthogonal
decomposition $u = \bar u + \widetilde u$ and  $\theta = \bar \theta + \widetilde \theta$, apply
various anisotropic inequalities in the previous section and make use of the fine properties of 
$\widetilde u$ and $\widetilde \theta$. More details are given in the proof of Theorem \ref{main1}.

\vskip .1in 
\begin{proof}[Proof of Theorem \ref{main1}]
We define the 
natural energy functional, 
$$
E(t): = \sup_{0\le \tau\le t} \big(\| u(\tau)\|_{H^2}^2+\| \theta(\tau)\|_{H^2}^2\big) + \nu \int_0^t \|\p_1 u(\tau)\|_{H^2}^2 \,d\tau+\kappa \int_0^t \|\p_1 \theta(\tau)\|_{H^2}^2 \,d\tau.
$$
Our main efforts are devoted to proving that, for a constant $C$ uniform for all $t>0$, 
\ben \label{non}
E(t) \le E(0) + C\, E(t)^{2}+ C\, E(t)^{3}.
\een
Once \eqref{non} is established, the bootstrapping argument implies that, if 
	$$
	E(0) = \|u_0\|_{H^2}^2+\|\theta_0\|_{H^2}^2 \le \varepsilon^2, \quad \text{for}~\varepsilon^2< \min\{\frac{1}{16C},\frac{1}{8\sqrt C}\},  
	$$
then  $E(t)$ admits the desired uniform global bound $E(t)  \le C\, \varepsilon^2$. 	
To initiate the bootstrapping argument, we make the ansatz
\ben \label{ansatz}
E(t)\leq \min\{\frac{1}{4C}, \frac{1}{2\sqrt{C}}\}. 
\een 
\eqref{non} will allow us to conclude that $E(t)$ actually admits an even smaller bound. In fact, if \eqref{ansatz} holds, then \eqref{non} implies 
\beno
E(t) \le E(0) + \frac{1}{4} E(t)+ \frac{1}{4} E(t),
\eeno
or 
$$
E(t)\leq 2E(0)\leq 2\varepsilon^2 \le \frac12\, \min\{\frac{1}{4C}, \frac{1}{2\sqrt{C}}\}
$$
with the bound being half of the one in (\ref{ansatz}). 
The bootstrapping argument then asserts that $E(t)$ is bounded uniformly for all time, 
\beq\label{nn}
E(t)  \le C\, \varepsilon^2.
\eeq

Then we can deduce the global existence as well as the stability result from the global bound in (\ref{nn}). Now we show  (\ref{non}). A $L^2$-estimate yields
\ben
\|u(t)\|_{L^2}^2+	\|\theta(t)\|_{L^2}^2 +  2\nu \int_0^t \|\p_1 u(\tau)\|_{L^2}^2\,d\tau +2\kappa \int_0^t \|\p_1 \theta(\tau)\|_{L^2}^2 \,d\tau  \notag \\
 = \|u_0\|_{L^2}^2+\|\theta_0\|_{L^2}^2.~~\label{l2-sec3} 
\een 
To estimate the $H^1$-norm, we make use of the vorticity equation associated with the velocity equation in  \eqref{B1},
$$
\p_t \om + u\cdot\na \om =\p_1\theta,
$$ 
where $\omega = \na\times u$.
Taking the inner product of $(\om,\Delta \theta)$ with the equations of vorticity and temperature,  we have, due to $\nabla\cdot u=0$,
\ben
\frac12\frac{d}{dt} (\|\om(t)\|_{L^2}^2+\|\nabla\theta(t)\|_{L^2}^2) +  \nu \|\p_1  \om\|_{L^2}^2+ \kappa \|\p_1 \na \theta\|_{L^2}^2 \notag \\
   =- \int \na (u \cdot \na \theta)\cdot \na \theta\,dx := M,~~\label{h1-sec3}  
\een
where we have used 
$$
\int\p_1\theta\,\om\,dx=-\int\theta\,\p_1\om\,dx=-\int\theta\,\Delta u_2\,dx=-\int\Delta\theta\, u_2\,dx.
$$
We further write $M$ in \eqref{h1-sec3} as
 
\beno
M&=&-\int \na (u \cdot \na \theta)\cdot \na \theta\,dx \\
&=& -\sum_{i,j=1}^2 \int \p_j(u_i \p_i\theta)\, \p_j \theta\,dx\\
&=& -\sum_{i,j=1}^2 \int \p_j u_i \p_i\theta\, \p_j \theta\,dx -\sum_{i,j=1}^2 \int u_i 
\p_i\p_j\theta\, \p_j \theta\,dx\\
&=& - \int \p_1 u_1 \,\p_1\theta\,\p_1 \theta\,dx - \int \p_1 u_2 \,\p_2\theta\,\p_1 \theta\,dx \\
&\quad& - \int \p_2 u_1\, \p_1\theta\,\p_2 \theta\,dx - \int \p_2 u_2 \,\p_2\theta\,\p_2 \theta\,dx \\
&:=& M_1+M_2+M_3+M_4.
\eeno
Here,  due to $\na\cdot u=0$, we have used   
$$
\sum_{i,j=1}^2 \int u_i \p_i\p_j\theta\, \p_j \theta\,dx =0. 
$$
 
By Lemma \ref{goo1}, Lemma \ref{goo4}, Lemma \ref{goo2} and Young's inequality,
\beno
M_1
&=& - \int \p_1 \widetilde{u}_1\, \p_1\widetilde{\theta}\,\p_1 \theta\,dx\\
&\leq& C \|\p_1 \theta\|_{L^2}\|\p_1\widetilde{\theta}\|_{L^2}^{\frac12}\|\p_1\p_1\widetilde{\theta}\|_{L^2}^{\frac12}\|\p_1\widetilde{u}_1\|_{L^2}^{\frac12}\|\p_2\p_1\widetilde{u}_1\|_{L^2}^{\frac12}\\
&\leq& C\|\p_1 \theta\|_{L^2}^2(\|\p_1{u}\|_{L^2}^2+\|\p_1{\theta}\|_{L^2}^2) +\delta(\|\p_1{\nabla u}\|_{L^2}^2+\|\p_1{\nabla\theta}\|_{L^2}^2),
\eeno
where $\delta>0$ is a small fixed constant to be specified later. 
Similarly, 
\beno
M_2
&=& - \int \p_1 \widetilde{u}_2\, \p_2\theta\,\p_1 \widetilde{\theta}\,dx\\
&\leq& C \|\p_2 \theta\|_{L^2}\|\p_1\widetilde{\theta}\|_{L^2}^{\frac12}\|\p_1\p_1\widetilde{\theta}\|_{L^2}^{\frac12}\|\p_1\widetilde{u}_2\|_{L^2}^{\frac12}\|\p_2\p_1\widetilde{u}_2\|_{L^2}^{\frac12}\\
&\leq& C\|\p_2 \theta\|_{L^2}^2(\|\p_1{u}\|_{L^2}^2+\|\p_1{\theta}\|_{L^2}^2) +\delta(\|\p_1{\nabla u}\|_{L^2}^2+\|\p_1{\nabla\theta}\|_{L^2}^2).
\eeno
To deal with $M_3$, we invoke the decompositions $u = \bar u + \widetilde u$ and  $\theta = \bar \theta + \widetilde \theta$
to write it into four terms,
\beno
M_3
&=& - \int \p_2 u_1 \,\p_1\theta\,\p_2 \theta\,dx \\
&=&- \int \p_2 \bar{u}_1\, \p_1\widetilde{\theta}\,\p_2 \bar{\theta}\,dx- \int \p_2 \bar{u}_1\, \p_1\widetilde{\theta}\,\p_2 \widetilde{\theta}\,dx \\
&\quad&- \int \p_2 \widetilde{u}_1 \,\p_1\widetilde{\theta}\,\p_2 \bar{\theta}\,dx- \int \p_2 \widetilde{u}_1\, \p_1\widetilde{\theta}\,\p_2 \widetilde{\theta}\,dx \\
&:=& M_{31}+M_{32}+M_{33}+M_{34}.
\eeno
According to Lemma \ref{goo1}, it is easy to see $M_{31}=0$. To bound $M_{32}$ and $M_{33}$, we use Lemma \ref{goo1}, Lemma \ref{goo4}, Lemma \ref{goo2} and Young's inequality to obtain
\beno
M_{32}
&=& - \int \p_2 \bar{u}_1\, \p_1\widetilde{\theta}\,\p_2 \widetilde{\theta}\,dx \\
&\leq& C \|\p_2 \bar{u}_1\|_{L^2}\|\p_1\widetilde{\theta}\|_{L^2}^{\frac12}\|\p_2\p_1\widetilde{\theta}\|_{L^2}^{\frac12}\|\p_2{\widetilde{\theta}}\|_{L^2}^{\frac12}\|\p_1\p_2{\widetilde{\theta}}\|_{L^2}^{\frac12}\\
&\leq& C \|\p_2 {u}\|_{L^2}\|\p_1\widetilde{\theta}\|_{L^2}^{\frac12}\|\p_1\p_2{\widetilde{\theta}}\|_{L^2}^{\frac32}\\
&\leq& C\|\p_2 u\|_{L^2}^4\|\p_1\theta\|_{L^2}^2+\delta\|\p_1{\nabla\theta}\|_{L^2}^2
\eeno
and 
\beno
M_{33}
&=& - \int \p_2 \widetilde{u}_1\, \p_1\widetilde{\theta}\,\p_2 \bar{\theta}\,dx \\
&\leq& C \|\p_2 \bar{\theta}\|_{L^2}\|\p_2\widetilde{u}_1\|_{L^2}^{\frac12}\|\p_1\p_2\widetilde{u}_1\|_{L^2}^{\frac12}\|\p_1\widetilde{\theta}\|_{L^2}^{\frac12}\|\p_2\p_1\widetilde{\theta}\|_{L^2}^{\frac12}\\
&\leq& C \|\p_2 {\theta}\|_{L^2}\|\p_1{\theta}\|_{L^2}^{\frac12}\|\p_1\p_2{\widetilde{u}_1}\|_{L^2}\|\p_1\p_2{\widetilde{\theta}}\|_{L^2}^{\frac12}\\
&\leq& C\|\p_2 \theta\|_{L^2}^4\|\p_1\theta\|_{L^2}^2+\delta(\|\p_1{\nabla u}\|_{L^2}^2+\|\p_1{\nabla\theta}\|_{L^2}^2).
\eeno
$M_{34}$ can be similarly bounded as $M_{32}$. For $M_4$, we also write it as 
\beno
M_4
&=&  \int \p_1 u_1\,\p_2\theta\,\p_2 \theta\,dx \\
&=& 2\int \p_1 \widetilde{u}_1\, \p_2\widetilde{\theta}\,\p_2 \bar{\theta}\,dx+ \int \p_1 \widetilde{u}_1\, \p_2\widetilde{\theta}\,\p_2 \widetilde{\theta}\,dx \\
&:=& M_{41}+M_{42}.
\eeno
Similar as $M_{33}$, we can bound $M_{41}$ and $M_{42}$ by
\beno
M_{41},~M_{42}
\leq C\|\p_2 \theta\|_{L^2}^4\|\p_1u\|_{L^2}^2+\delta(\|\p_1{\nabla u}\|_{L^2}^2+\|\p_1{\nabla\theta}\|_{L^2}^2).
\eeno
Collecting the estimates for $M$ and taking $\delta>0$ to be small, say 
$$
\delta \le \frac1{16} \min\{\nu, \eta\}, 
$$
we obtain 
\ben
&&\quad\|\om(t)\|_{L^2}^2+\|\nabla\theta(t)\|_{L^2}^2 + \nu\int_0^t \|\p_1  \om(\tau)\|_{L^2}^2\tau+\kappa\int_0^t  \|\p_1 \na \theta(\tau)\|_{L^2}^2\,d\tau \notag \\
&&\leq C\int_0^t (\|\nabla \theta\|_{L^2}^2+\|\nabla \theta\|_{L^2}^4+\|\nabla u\|_{L^2}^4)\times(\|\p_1u\|_{L^2}^2+\|\p_1\theta\|_{L^2}^2)\,d\tau. 
\label{h1'-sec3}
\een
Next we estimate the $H^2$-norm of $(u, \theta)$. We take the inner 
product of $\Delta\om$ and $\Delta^2\theta$ with  
the equations of vorticity and temperature, respectively. Due to $\nabla\cdot u=0$ and after integrating by parts, we have
\ben
&&\quad\frac12\frac{d}{dt} (\|\nabla\om(t)\|_{L^2}^2+\|\Delta\theta(t)\|_{L^2}^2) +  \nu \|\p_1  \nabla\om\|_{L^2}^2+ \kappa \|\p_1 \Delta \theta\|_{L^2}^2 \notag\\
&&=- \int \na \om \cdot \na u \cdot \na \om\,dx- \int \Delta(u \cdot \na \theta) \Delta \theta\,dx .\label{h2-sec3}
\een
For the first term on the right hand side, we can decompose it as
\ben
N&:=&- \int \na \om \cdot \na u \cdot \na \om\,dx\notag\\
 &=& -\int \p_1 u_1 \, (\p_1\om)^2\,dx - \int \p_1 u_2\, \p_1\om \, \p_2\om\,dx \notag \\
&& - \int \p_2 u_1\,\p_1\om \, \p_2\om\,dx - \int \p_2 u_2\, (\p_2 \om)^2\,dx\notag\\
&:=& N_1+ N_2+N_3+N_4.\notag
\een
$N_1$ and $N_2$ can be bounded directly. According to Lemma \ref{goo1}, $\p_1 \bar u=0$ and $\p_1 u= \p_1 \widetilde u$. By Lemma \ref{goo4}, 
\beno 
N_1 &=& -\int \p_1 u_1 \, \p_1\om\,\p_1\widetilde{\om}\,dx\\
&\le& C\, \|\p_1 u_1\|_{L^2}^{\frac12} \,\|\p_1 \p_1 u_1\|_{L^2}^{\frac12} 
\,\|\p_1 \om\|^{\frac12}_{L^2}\, \|\p_2\p_1\om\|_{L^2}^{\frac12}\,\|\p_1\widetilde\om\|_{L^2}\\
&\le& C\, \|\p_1 u\|_{H^1} \, \|\p_1 u\|_{H^2} \|\p_1\nabla\om\|_{L^2}\\
&\le& C\, \|\p_1 u\|_{H^1}^2 \, \|\p_1 u\|_{H^2}^2+\delta' \|\p_1\nabla\om\|_{L^2}^2\\
\eeno 
and 
\beno 
N_2 &\le&  C\, \|\p_1 u_2\|_{L^2}^{\frac12} \,\|\p_1 \p_1 u_2\|_{L^2}^{\frac12} 
\,\|\p_1 \widetilde\om\|^{\frac12}_{L^2}\, \|\p_2\p_1\widetilde\om\|_{L^2}^{\frac12}\,\|\p_2\om\|_{L^2}\\
&\le& C\, \|u\|_{H^2}^2 \, \|\p_1 u\|^2_{H^2}+\delta' \|\p_1\nabla\om\|_{L^2}^2,
\eeno 
where $\delta'>0$ is a small but fixed parameter. The estimate of $N_3$ is slightly more delicate. 
\beno 
N_3 &=& - \int  \p_2 u_1\,\p_1\om \, \p_2\om\,dx \\
&=& - \int \p_2 (\bar u_1 + \widetilde u_1) \, \p_1 \widetilde \om\, \p_2(\bar \om + \widetilde \om)\,dx \\
&=& - \int \p_2\bar u_1\,	\, \p_1 \widetilde \om\,\p_2 \bar \om\,dx - \int \p_2\bar u_1\,	\, \p_1 \widetilde \om\,\p_2 \widetilde \om\,dx \\
&&- \int \p_2\widetilde u_1\,	\, \p_1 \widetilde \om\,\p_2 \bar \om\,dx - \int \p_2\widetilde u_1\,	\, \p_1 \widetilde \om\,\p_2 \widetilde \om\,dx\\
&:=& N_{31} + N_{32} + N_{33} + N_{34}. 
\eeno 
The first term $N_{31}$ is clearly zero, 
$$
N_{31} = - \int_{\mathbb R}  \p_2\bar u_1\,\p_2 \bar \om\,\int_{\mathbb T} \p_1 \widetilde \om_1\,dx_1\,dx_2 =0.
$$
To bound $N_{32}$ and $N_{33}$, we first use (\ref{an2}) of Lemma \ref{goo4} and then Lemma \ref{goo2} to obtain 
\beno 
N_{32} &\le& C\, \|\p_2 \bar u_1\|_{L^2} \, \|\p_1 \widetilde \om\|_{L^2}^{\frac12} \, \|\p_2\p_1 \widetilde \om\|_{L^2}^{\frac12}\, \|\p_2 \widetilde \om\|_{L^2}^{\frac12}   \|\p_1\p_2 \widetilde \om\|_{L^2}^{\frac12} \\
&\le&\,C\, \|\p_2 \bar u_1\|_{L^2} \,\|\p_1 \widetilde \om\|_{L^2}^{\frac12} \,\|\p_1\p_2 \widetilde \om\|_{L^2}^{\frac32} \\
&\le& C\, \|u\|_{H^1}^4 \, \|\p_1 u\|_{H^2}^2+\delta' \|\p_1\nabla\om\|_{L^2}^2
\eeno 
and 
\beno
N_{33} &\le& C\, \|\p_2 \bar \om\|_{L^2}\, \|\p_2 \widetilde u_1\|^{\frac12}_{L^2} \, \|\p_1 \p_2 \widetilde u_1\|^{\frac12}_{L^2}\,\|\p_1 \widetilde \om\|_{L^2}^{\frac12} \, \|\p_2\p_1 \widetilde \om\|_{L^2}^{\frac12}\\
&\le& C\, \|u\|_{H^2}^2 \, \|\p_1 u\|_{H^2}^2+\delta' \|\p_1\nabla\om\|_{L^2}^2,
\eeno 
$N_{34}$ can be similarly bounded as $N_{32}$. $N_4$ can also be bounded similarly. 
\beno
N_4 &=& - \int \p_1 \widetilde u_1 \, (\p_2 \bar \om + \p_2\widetilde \om)^2\,dx \\
&=& -2 \int \p_1 \widetilde u_1 \,\p_2 \bar \om\,\p_2\widetilde \om\,dx - \int \p_1 \widetilde u_1 \, (\p_2\widetilde \om)^2 \,dx\\
&\le& C\, (\|\p_2 \bar \om\|_{L^2} + \|\p_2 \widetilde \om\|_{L^2})\|\p_1 \widetilde u_1\|_{L^2}^{\frac12}\, 
\|\p_2\p_1 \widetilde u_1\|_{L^2}^{\frac12}\, \|\p_2\widetilde \om\|_{L^2}^{\frac12} \, \|\p_1\p_2\widetilde \om\|_{L^2}^{\frac12} \\
&\le& C\,\|u\|_{H^2}^2 \,\|\p_1 u\|_{H^2}^2+\delta' \|\p_1\nabla\om\|_{L^2}^2.
\eeno 
Thus we have shown that 
\beq 
|N| \le C\, \,(\|u\|_{H^2}^2+\|u\|_{H^2}^4) \,\|\p_1 u\|_{H^2}^2+6\delta' \|\p_1\nabla\om\|_{L^2}^2.\label{nb-sec3}
\eeq
For the last term of \eqref{h2-sec3}, we can write it as 
\beno
- \int \Delta(u \cdot \na \theta)\, \Delta \theta\,dx \hspace{-.5cm} &&= - \int \Delta u \cdot \na \theta\,\Delta \theta\,dx- 2\int \nabla u \cdot \na^2 \theta \,\Delta \theta\,dx\\
&&=- \int \Delta u_1\, \p_1  \theta\, \Delta \theta\,dx- \int \Delta u_2 \,\p_2  \theta\, \Delta \theta\,dx\\
&&\quad- 2\int \p_1 u_1 \,\p_1\p_1 \theta \,\Delta \theta\,dx- 2\int \p_1 u_2 \,\p_2\p_1 \theta \,\Delta \theta\,dx\\
&&\quad- 2\int \p_2 u_1 \,\p_1\p_2 \theta \,\Delta \theta\,dx- 2\int \p_2 u_2 \,\p_2\p_2 \theta \,\Delta \theta\,dx\\
&&:= P_1+P_2+P_3+P_4+P_5+P_6.
\eeno
According to Lemma \ref{goo1}, we can divide $P_1$ into four terms, 
\beno
P_1&=& - \int \Delta \bar{u}_1\, \p_1  \widetilde{\theta} \,\Delta \bar{\theta}\,dx- \int \Delta \bar{u}_1 \,\p_1  \widetilde{\theta} \,\Delta \widetilde{\theta}\,dx\\
&\quad&- \int \Delta \widetilde{u}_1 \,\p_1  \widetilde{\theta} \,\Delta \bar{\theta}\,dx- \int \Delta \widetilde{u}_1 \,\p_1  \widetilde{\theta} \,\Delta \widetilde{\theta}\,dx\\
&:=& P_{11}+P_{12}+P_{13}+P_{14}.
\eeno
It is clear that $P_{11}=0$. For $P_{12}$, we can bound it using Lemma \ref{goo1}, Lemma \ref{goo4} and Lemma \ref{goo2},
\beno
P_{12}
&\leq& C \|\Delta \bar{u}_1\|_{L^2}\|\p_1\widetilde{\theta}\|_{L^2}^{\frac12}\|\p_2\p_1\widetilde{\theta}\|_{L^2}^{\frac12}\|\Delta\widetilde{\theta}\|_{L^2}^{\frac12}\|\p_1\Delta\widetilde{\theta}\|_{L^2}^{\frac12}\\
&\leq& C \|\Delta {u}\|_{L^2}\|\p_1{\theta}\|_{L^2}^{\frac12}\|\p_1\nabla{\theta}\|_{L^2}^{\frac12}\|\p_1\Delta\widetilde{\theta}\|_{L^2}\\
&\leq& C\|\Delta u\|_{L^2}^2(\|\p_1\theta\|_{L^2}^2+\|\p_1\nabla\theta\|_{L^2}^2)+\delta'\|\p_1{\Delta\theta}\|_{L^2}^2.
\eeno
Similarly, $P_{14}$ shares the same bounded with $P_{12}$. For $P_{13}$, we can bound it by
\beno
P_{13}
&\leq& C \|\Delta \bar{\theta}\|_{L^2}\|\Delta\widetilde{u}_1\|_{L^2}^{\frac12}\|\p_1\Delta\widetilde{u}_1\|_{L^2}^{\frac12}\|\p_1\widetilde{\theta}\|_{L^2}^{\frac12}\|\p_2\p_1\widetilde{\theta}\|_{L^2}^{\frac12}\\
&\leq& C \|\Delta {\theta}\|_{L^2}\|\p_1{\theta}\|_{L^2}^{\frac12}\|\p_1\Delta\widetilde{u}\|_{L^2}\|\p_1\Delta\widetilde{\theta}\|_{L^2}^{\frac12}\\
&\leq& C\|\Delta \theta\|_{L^2}^4\|\p_1\theta\|_{L^2}^2+\delta'(\|\p_1{\Delta u}\|_{L^2}^2+\|\p_1{\Delta\theta}\|_{L^2}^2).
\eeno
Therefore, 
\beno
P_{1}
\leq C(\|\Delta u\|_{L^2}^2+\|\Delta \theta\|_{L^2}^4)\times\|\p_1\theta\|_{H^1}^2+\delta'(\|\p_1{\Delta u}\|_{L^2}^2+\|\p_1{\Delta\theta}\|_{L^2}^2).
\eeno
According to the relation $\Delta u_2=\p_1\om$, we can decompose $P_2$ as follows, 
\beno
P_2&=& - \int \p_1\om \,\p_2  {\theta} \,\Delta {\theta}\,dx\\
&=&- \int \p_1 \widetilde{\om}\, \p_2  \bar{\theta} \,\Delta \bar{\theta}\,dx- \int \p_1 \widetilde{\om} \,\p_2 \bar{\theta} \,\Delta \widetilde{\theta}\,dx\\
&\quad&- \int \p_1 \widetilde{\om}\, \p_2  \widetilde{\theta} \,\Delta \bar{\theta}\,dx- \int \p_1 \widetilde{\om} \,\p_2 \widetilde{\theta} \,\Delta \widetilde{\theta}\,dx\\
&:=& P_{21}+P_{22}+P_{23}+P_{24}.
\eeno
Clearly,  $P_{21}=0$.  Making use of Lemma \ref{goo1}, Lemma \ref{goo4} and Lemma \ref{goo2}, we  obtain
\beno
P_{22},P_{24}
&\leq& C \|\p_2 \bar{\theta}\|_{L^2}\|\p_1\widetilde{\om}\|_{L^2}^{\frac12}\|\p_2\p_1\widetilde{\om}\|_{L^2}^{\frac12}\|\Delta\widetilde{\theta}\|_{L^2}^{\frac12}\|\p_1\Delta\widetilde{\theta}\|_{L^2}^{\frac12}\\
&\leq& C \|\nabla\theta \|_{L^2}\|\p_1{\om}\|_{L^2}^{\frac12}\|\p_1\nabla\widetilde{\om}\|_{L^2}^{\frac12}\|\p_1\Delta\widetilde{\theta}\|_{L^2}\\
&\leq& C\|\nabla \theta\|_{L^2}^4\|\p_1\om\|_{L^2}^2+\delta'(\|\p_1{\nabla\om }\|_{L^2}^2+\|\p_1{\Delta\theta}\|_{L^2}^2)
\eeno
and
\beno
P_{23}
&\leq& C \|\Delta \bar{\theta}\|_{L^2}\|\p_1\widetilde{\om}\|_{L^2}^{\frac12}\|\p_2\p_1\widetilde{\om}\|_{L^2}^{\frac12}\|\p_2\widetilde{\theta}\|_{L^2}^{\frac12}\|\p_1\p_2\widetilde{\theta}\|_{L^2}^{\frac12}\\
&\leq& C \|\Delta\theta \|_{L^2}\|\p_1{\om}\|_{L^2}^{\frac12}\|\p_1\nabla\widetilde{\om}\|_{L^2}^{\frac12}\|\p_1\Delta\widetilde{\theta}\|_{L^2}\\
&\leq& C\|\Delta \theta\|_{L^2}^4\|\p_1\om\|_{L^2}^2+\delta'(\|\p_1{\nabla\om }\|_{L^2}^2+\|\p_1{\Delta\theta}\|_{L^2}^2).
\eeno
Thus,
\beno
P_{2}
\leq C(\|\nabla u\|_{L^2}^4+\|\Delta \theta\|_{L^2}^4)\times\|\p_1\om\|_{L^2}^2+\delta'(\|\p_1{\Delta u}\|_{L^2}^2+\|\p_1{\Delta\theta}\|_{L^2}^2).
\eeno
For $P_3$, we can bound it by
\beno
P_{3}&=&-2\int \p_1\widetilde{u}_1\,\p_1\p_1\widetilde{\theta}\,\Delta\theta\,dx\\
&\leq& C \|\Delta {\theta}\|_{L^2}\|\p_1\widetilde{u}_1\|_{L^2}^{\frac12}\|\p_2\p_1\widetilde{u}_1\|_{L^2}^{\frac12}\|\p_1\p_1\widetilde{\theta}\|_{L^2}^{\frac12}\|\p_1\p_1\p_1\widetilde{\theta}\|_{L^2}^{\frac12}\\
&\leq& C \|\Delta\theta \|_{L^2}\|\p_1{u}\|_{L^2}^{\frac12}\|\p_1\Delta\widetilde{u}\|_{L^2}^{\frac12}\|\p_1\Delta\widetilde{\theta}\|_{L^2}\\
&\leq& C\|\Delta \theta\|_{L^2}^4\|\p_1u\|_{L^2}^2+\delta'(\|\p_1{\Delta u }\|_{L^2}^2+\|\p_1{\Delta\theta}\|_{L^2}^2).
\eeno
$P_{4}$ can be bounded in the same way. Using Lemma \ref{goo1}, we can write $P_5$ as
\beno
P_5&=& - 2\int \p_2 \bar{u}_1 \,\p_1 \p_2 \widetilde{\theta}\, \Delta \bar{\theta}\,dx- 2\int \p_2 \bar{u}_1 \,\p_1\p_2  \widetilde{\theta} \,\Delta \widetilde{\theta}\,dx\\
&\quad&- 2\int \p_2 \widetilde{u}_1 \,\p_1\p_2  \widetilde{\theta} \,\Delta \bar{\theta}\,dx- 2\int \p_2 \widetilde{u}_1 \,\p_1\p_2  \widetilde{\theta} \,\Delta \widetilde{\theta}\,dx\\
&:=& P_{51}+P_{52}+P_{53}+P_{54}.
\eeno
It is easy to check that $P_{51}=0$. $P_{52}$ and $P_{54}$ can be bounded by
\beno
P_{52},P_{54}
\leq C\|\nabla u\|_{L^2}^4\|\p_1\nabla\theta\|_{L^2}^2+\delta'\|\p_1{\Delta\theta}\|_{L^2}^2,
\eeno
and $P_{53}$ have the bound 
\beno
P_{53}
\leq C\|\Delta \theta\|_{L^2}^4\|\p_1\nabla u\|_{L^2}^2+\delta'(\|\p_1{\Delta u}\|_{L^2}^2+\|\p_1{\Delta\theta}\|_{L^2}^2).
\eeno
Finally we estimate $P_6$, which can be written as 
\beno
P_6&=& 2\int \p_1u_1\,\p_2\p_2\theta\,\Delta\theta\,dx\\
&=&2\int \p_1 \widetilde{u}_1 \,\p_2 \p_2 \bar{\theta} \,\Delta \bar{\theta}\,dx+ 2\int \p_1 \widetilde{u}_1 \,\p_2\p_2  \bar{\theta} \,\Delta \widetilde{\theta}\,dx\\
&\quad&+ 2\int \p_1 \widetilde{u}_1 \,\p_2\p_2  \widetilde{\theta} \,\Delta \bar{\theta}\,dx+ 2\int \p_1 \widetilde{u}_1 \,\p_2\p_2  \widetilde{\theta} \,\Delta \widetilde{\theta}\,dx\\
&:=& P_{61}+P_{62}+P_{63}+P_{64}.
\eeno
As in the estimate of $P_1$, we have $P_{61}=0$ and 
\beno
P_{62},P_{63},P_{64}
\leq C\|\Delta \theta\|_{L^2}^4\|\p_1 u\|_{L^2}^2+\delta'(\|\p_1{\Delta u}\|_{L^2}^2+\|\p_1{\Delta\theta}\|_{L^2}^2).
\eeno
Inserting \eqref{nb-sec3} and the estimates for $P_1$ through $P_6$ in \eqref{h2-sec3}, and choosing $\delta'>0$  
sufficiently small, we obtain
\ben
&& \quad\|\nabla\om(t)\|_{L^2}^2+\|\Delta\theta(t)\|_{L^2}^2 +  \nu\int_0^t \|\p_1  \nabla\om(\tau)\|_{L^2}^2\,d\tau+ \kappa\int_0^t \|\p_1 \Delta \theta(\tau)\|_{L^2}^2\,d\tau \notag\\
&&\leq C\int_0^t\, (\|\p_1 u\|_{H^2}^2+\|\p_1 \theta\|_{H^2}^2)\times (\| u\|_{H^2}^2+\| u\|_{H^2}^4+\| \theta\|_{H^2}^2+\| \theta\|_{H^2}^4)\,d\tau.\label{h2'-sec3}
\een
Combining \eqref{h2'-sec3},  \eqref{l2-sec3} and \eqref{h1'-sec3} leads to the desired inequality in  \eqref{non}.
This completes the proof of Theorem \ref{main1}. 
\end{proof}

\vskip .3in 
\section{Proof of Theorem \ref{main2}}
\label{sec4}

This section proves Theorem \ref{main2}. We work with the equations of $(\widetilde u, \widetilde \theta)$ 
and make use of the properties of the orthogonal decomposition and various anisotropic inequalities. 

\vskip .1in 
\begin{proof}[Proof of Theorem \ref{main2}]. 
We first write the equation of $(\bar u, \bar \theta)$. {Making use of Lemma \ref{goo1}, we have $\p_1\bar u=0$ and 
\beno
\overline{u\cdot \nabla \bar{u}}=\overline{u_1\p_1  \bar{u}}+\overline{u_2\p_2  \bar{u}}={\bar{u}_2\p_2  \bar{u}}.
\eeno
Since $\nabla \cdot u=0$ in $\Omega$,  there exists a stream function $\psi$ such that 
$$
u=\nabla^{\perp}\psi: = (-\p_2 \psi, \p_1\psi).
$$
Then
\beno
\bar{u}_2=\overline{\p_1\psi}=0
\eeno
and 
\beq\label{w5}
\overline{u\cdot \nabla \bar{u}}=0.
\eeq
}  Taking the average of (\ref{B1}) and making use of (\ref{w5}) yield
\beq\label{bar-boussinesq}
\begin{cases}
\partial_{t} \bar{u}+\overline{u\cdot \nabla \widetilde{u}}+\left(\begin{array}{c}
	{0} \\
	{\partial_{2} \bar{p}}
	\end{array}\right)
	=\left(\begin{array}{c}
		{0} \\
		{\bar{\theta}}
	\end{array}\right), \\
\partial_t \bar{\theta}+\overline{u\cdot \nabla \widetilde{\theta}}=0.
\end{cases}
\eeq
Taking the difference of (\ref{B1}) and (\ref{bar-boussinesq}), we find 
\beq\label{tilde-boussinesq}
\begin{cases}
\p_{t} \widetilde{u}+\widetilde{u\cdot \nabla \widetilde{u}}+u_2\p_2\bar{u}-\nu\p_1^2\widetilde{u}+\nabla\widetilde{p}
		=\widetilde{\theta}e_2,\\
\p_{t} \widetilde{\theta}+\widetilde{u\cdot \nabla \widetilde{\theta}}+u_2\p_2\bar{\theta}-{ \kappa\p_1^2\widetilde{\theta}}+\widetilde{u}_2
		=0.
\end{cases}
\eeq
The $L^2$-estimate gives
\beno
&\quad&\frac12\frac{d}{dt}(\|\widetilde{u}(t)\|_{L^2}^2+\|\widetilde{\theta}(t)\|_{L^2}^2)+\nu\|\p_1\widetilde{u}\|_{L^2}^2+\kappa\|\p_1\widetilde{\theta}\|_{L^2}^2\\
&=&-\int \widetilde{u\cdot \nabla \widetilde{u}}\cdot\widetilde{u}\,dx-\int u_2\,\p_2\bar{u}\cdot\widetilde{u}\,dx-\int \widetilde{u\cdot \nabla \widetilde{\theta}}~\widetilde{\theta}\,dx-\int u_2\,\p_2\bar{\theta}~\widetilde{\theta}\,dx\\
&:=& A_1+A_2+A_3+A_4.
\eeno
For $A_1$, according to the divergence-free condition of $u$ and Lemma \ref{goo1}, we have 
\beno
A_1&=&-\int \widetilde{u\cdot \nabla \widetilde{u}}\cdot\widetilde{u}\,dx=-\int {u\cdot \nabla \widetilde{u}}\cdot\widetilde{u}\,dx+\int \overline{u\cdot \nabla \widetilde{u}}\cdot\widetilde{u}\,dx=0.
\eeno
Similarly, $A_3=0$. Then we estimate $A_2$. By Lemma \ref{goo1}, Lemma \ref{goo4} and Lemma \ref{goo2}, 
\beno
A_2&=&-\int \widetilde{u}_2\,\p_2\bar{u}\cdot\widetilde{u}\,dx\\
&\leq& C\|\p_2\bar{u}\|_{L^2}\|\widetilde{u}_2\|_{L^2}^{\frac12}\|\p_2\widetilde{u}_2\|_{L^2}^{\frac12}\|\widetilde{u}\|_{L^2}^{\frac12}\|\p_1\widetilde{u}\|_{L^2}^{\frac12}\\
&\leq& C\|\p_2\bar{u}\|_{L^2}\|\p_1\widetilde{u}_2\|_{L^2}^{\frac12}\|\p_1\widetilde{u}_1\|_{L^2}^{\frac12}\|\p_1\widetilde{u}\|_{L^2}^{\frac12}\|\p_1\widetilde{u}\|_{L^2}^{\frac12}\\
&\leq& C\|u\|_{H^1}\|\p_1\widetilde{u}\|_{L^2}^2.
\eeno
Similarly, 
\beno
A_4=-\int \widetilde{u}_2\,\p_2\bar{\theta}~\widetilde{\theta}\,dx\leq C\|\theta\|_{H^1}(\|\p_1\widetilde{u}\|_{L^2}^2+\|\p_1\widetilde{\theta}\|_{L^2}^2).
\eeno
Collecting the estimates for $A_1$ through $A_4$, we obtain 
\beno
\frac{d}{dt}(\|\widetilde{u}(t)\|_{L^2}^2+\|\widetilde{\theta}(t)\|_{L^2}^2)+(2 \nu-C\|(u,\theta)\|_{H^1})\|\p_1\widetilde{u}\|_{L^2}^2\\
+(2 \kappa-C\|(u,\theta)\|_{H^1})\|\widetilde{\theta}\|_{L^2}^2\leq 0.
\eeno
By Theorem \ref{main1},  if $\varepsilon>0$ is sufficiently small and
$\|u_0\|_{H_2}+\|\theta_0\|_{H_2}\leq\varepsilon$, then $\|(u(t),\theta(t))\|_{H^2}\leq C\varepsilon$ and
$$
2\nu - C\|(u(t),\theta(t))\|_{H^2}\geq \nu,\quad 2\kappa - C\|(u(t),\theta(t))\|_{H^2}\geq \kappa.
$$
Invoking the Poincar\'e type inequality in Lemma \ref{goo2} leads to the desired exponential decay for $\|(\widetilde{u},\widetilde{\theta})\|_{L^2}$, 
\ben\label{L2-decay}
\|\widetilde{u}(t)\|_{L^2}+\|\widetilde{\theta}(t)\|_{L^2}\leq  (\|{u}_0\|_{L^2}+\|{\theta}_0\|_{L^2})e^{-C_1t},
\een
where $C_1= C_1(\nu, \eta)>0$. We now turn to the exponential decay of $\|(\na \widetilde{u}(t), \na \widetilde{\theta}(t))\|_{L^2}$. Applying $\nabla$ to (\ref{tilde-boussinesq}) yields
\beq\label{d-tilde-boussinesq}
\begin{cases}
\p_{t} \nabla\widetilde{u}+\nabla(\widetilde{u\cdot \nabla \widetilde{u}})+\nabla(u_2\p_2\bar{u})-\nu\p_1^2\nabla\widetilde{u}+\nabla\nabla\widetilde{p}
=\nabla(\widetilde{\theta}e_2),\\
\p_{t} \nabla\widetilde{\theta}+\nabla(\widetilde{u\cdot \nabla \widetilde{\theta}})+\nabla(u_2\p_2\bar{\theta})-\kappa{\p_1^2\nabla\widetilde{\theta}}+\nabla\widetilde{u}_2
	=0.
\end{cases}
\eeq
Taking the $L^2$ inner product of system \eqref{d-tilde-boussinesq} with $(\nabla\widetilde{u},\,\nabla\widetilde{\theta})$, we have
\ben
&\quad&\frac12\frac{d}{dt}(\|\nabla\widetilde{u}(t)\|_{L^2}^2+\|\nabla\widetilde{\theta}(t)\|_{L^2}^2)+\nu\|\p_1\nabla\widetilde{u}\|_{L^2}^2+\kappa\|\p_1\nabla\widetilde{\theta}\|_{L^2}^2\notag\\
&=&-\int \nabla(\widetilde{u\cdot \nabla \widetilde{u}})\cdot\nabla\widetilde{u}\,dx-\int\nabla (u_2\p_2\bar{u})\cdot\nabla\widetilde{u}\,dx\notag\\
&\quad&-\int \nabla(\widetilde{u\cdot \nabla \widetilde{\theta}})\cdot\nabla\widetilde{\theta}\,dx-\int \nabla(u_2\p_2\bar{\theta})\cdot\nabla\widetilde{\theta}\,dx\notag\\
&:=& B_1+B_2+B_3+B_4.\label{H1-tilde-u}
\een
By Lemma \ref{goo1}, $B_1$ can be further written as  
\beno
B_1&=&-\int \nabla({u\cdot \nabla \widetilde{u}})\cdot\nabla\widetilde{u}\,dx +\int \nabla(\overline{u\cdot \nabla \widetilde{u}})\cdot\nabla\widetilde{u}\,dx\\
&=&-\int\p_1u_1\,\p_1\widetilde{u}\cdot\p_1\widetilde{u}\,dx+\int\p_1u_2\,\p_2\widetilde{u}\cdot\p_1\widetilde{u}\,dx\\
&\quad&-\int\p_2u_1\,\p_1\widetilde{u}\cdot\p_2\widetilde{u}\,dx+\int\p_2u_2\,\p_2\widetilde{u}\cdot\p_2\widetilde{u}\,dx\\
&:=& B_{11}+B_{12}+B_{13}+B_{14}.
\eeno
Using Lemma \ref{goo4} and Lemma \ref{goo2}, $B_{11}$ can be bounded by
\beno
B_{11}&\leq&C\|\p_1u_1\|_{L^2}\|\p_1\widetilde{u}\|_{L^2}^{\frac12}\|\p_2\p_1\widetilde{u}\|_{L^2}^{\frac12}\|\p_1\widetilde{u}\|_{L^2}^{\frac12}\|\p_1\p_1\widetilde{u}\|_{L^2}^{\frac12}\\
&\leq&C\|\p_1u_1\|_{L^2}\|\p_1\widetilde{u}\|_{L^2}^{\frac12}\|\p_2\p_1\widetilde{u}\|_{L^2}^{\frac12}\|\p_1\widetilde{u}\|_{L^2}^{\frac12}\|\p_1\p_1\widetilde{u}\|_{L^2}^{\frac12}\\
&\leq&C\|u\|_{H^1}\|\p_1\nabla\widetilde{u}\|_{L^2}^2.
\eeno
$B_{12}$ and $B_{13}$ can be bounded similarly and
\beno
B_{12},\,B_{13}\leq C\|u\|_{H^1}\|\p_1\nabla\widetilde{u}\|_{L^2}^2.
\eeno
For $B_{14}$, according to the divergence-free condition of $u$ and similar as $B_{11}$, we have 
\beno
B_{14}&=&{ -}\int\p_1u_1\,\p_2\widetilde{u}\cdot\p_2\widetilde{u}\,dx\,=\,\int\p_1\widetilde{u}_1\,\p_2\widetilde{u}\cdot\p_2\widetilde{u}\,dx\\
&\leq&C\|\p_2\widetilde{u}\|_{L^2}\|\p_1\widetilde{u}_1\|_{L^2}^{\frac12}\|\p_2\p_1\widetilde{u}_1\|_{L^2}^{\frac12}\|\p_2\widetilde{u}\|_{L^2}^{\frac12}\|\p_1\p_2\widetilde{u}\|_{L^2}^{\frac12}\\
&\leq&C\|u\|_{H^1}\|\p_1\nabla\widetilde{u}\|_{L^2}^2.
\eeno
Therefore,  $B_1$ is bounded by
\beno
|B_{1}|\leq C\|u\|_{H^1}\|\p_1\nabla\widetilde{u}\|_{L^2}^2.
\eeno
Similarly, we can bound $B_3$ by
\beno
|B_{3}|\leq C(\|u\|_{H^1}+\|\theta\|_{H^1})\times(\|\p_1\nabla\widetilde{u}\|_{L^2}^2+\|\p_1\nabla\widetilde{\theta}\|_{L^2}^2).
\eeno
To bound $B_2$, we first write it explicitly as 
\beno
B_2&=&-\int\nabla (u_2\p_2\bar{u})\cdot\nabla\widetilde{u}\,dx\\
&=&-\int\p_1 u_2\,\p_2\bar{u}\cdot\p_1\widetilde{u}\,dx-\int u_2\,\p_2\p_1\bar{u}\cdot\p_1\widetilde{u}\,dx\\
&\quad&-\int\p_2 u_2\,\p_2\bar{u}\cdot\p_2\widetilde{u}\,dx-\int u_2\,\p_2\p_2\bar{u}\cdot\p_2\widetilde{u}\,dx\\
&:=&B_{21}+B_{22}+B_{23}+B_{24}.
\eeno
By Lemma \ref{goo1}, Lemma \ref{goo4} and Lemma \ref{goo2}, 
\beno
B_{21}&=&-\int\p_1\widetilde{u}_2\,\p_2\bar{u}\cdot\p_1\widetilde{u}\,dx\\
&\leq&C\|\p_2\bar{u}\|_{L^2}\|\p_1\widetilde{u}_2\|_{L^2}^{\frac12}\|\p_2\p_1\widetilde{u}_2\|_{L^2}^{\frac12}\|\p_1\widetilde{u}\|_{L^2}^{\frac12}\|\p_1\p_1\widetilde{u}\|_{L^2}^{\frac12}\\
&\leq&C\|u\|_{H^1}\|\p_1\nabla\widetilde{u}\|_{L^2}^2.
\eeno
According to the definition of $\bar{u}$, 
\beno
B_{22}=-\int u_2\,\p_2\p_1\bar{u}\cdot\p_1\widetilde{u}\,dx=0.
\eeno
Similarly,  {$B_{23}$ and $B_{24}$} can be bounded by
\beno
B_{23}&=&\int\p_1\widetilde{u}_1\,\p_2\bar{u}\cdot\p_2\widetilde{u}\,dx\\
&\leq&C\|\p_2\bar{u}\|_{L^2}\|\p_1\widetilde{u}_1\|_{L^2}^{\frac12}\|\p_2\p_1\widetilde{u}_1\|_{L^2}^{\frac12}\|\p_2\widetilde{u}\|_{L^2}^{\frac12}\|\p_1\p_2\widetilde{u}\|_{L^2}^{\frac12}\\
&\leq&C\|u\|_{H^1}\|\p_1\nabla\widetilde{u}\|_{L^2}^2
\eeno
and 
\beno
B_{24}&=&-\int\widetilde{u}_2\,\p_2\p_2\bar{u}\cdot\p_2\widetilde{u}\,dx\\
&\leq&C\|\p_2\p_2\bar{u}\|_{L^2}\|\widetilde{u}_2\|_{L^2}^{\frac12}\|\p_2\widetilde{u}_2\|_{L^2}^{\frac12}\|\p_2\widetilde{u}\|_{L^2}^{\frac12}\|\p_1\p_2\widetilde{u}\|_{L^2}^{\frac12}\\
&\leq&C\|u\|_{H^2}\|\p_1\nabla\widetilde{u}\|_{L^2}^2.
\eeno
Thus we obtain the bound for $B_2$, 
\beno
|B_{2}|\leq C\|u\|_{H^2}\|\p_1\nabla\widetilde{u}\|_{L^2}^2.
\eeno
Similarly,  $B_4$ is bounded by
\beno
|B_{4}|\leq C\|\theta\|_{H^2}\times(\|\p_1\nabla\widetilde{u}\|_{L^2}^2+\|\p_1\nabla\widetilde{\theta}\|_{L^2}^2).
\eeno
Inserting the estimates for $B_1$ through $B_4$ in \eqref{H1-tilde-u},  we obtain
\beno
\frac{d}{dt}(\|\nabla\widetilde{u}(t)\|_{L^2}^2+\|\nabla\widetilde{\theta}(t)\|_{L^2}^2)+ (2\nu -C \|(u, \theta)\|_{H^2}) \|\p_1\nabla\widetilde{u}\|_{L^2}^2\\
+ (2\kappa -C \|(u, \theta)\|_{H^2}) \|\p_1\nabla\widetilde{\theta}\|_{L^2}^2\notag\leq 0.
\eeno
Choosing $\varepsilon$ sufficiently small and invoking the Poincar\'e type inequality in Lemma \ref{goo2}, we  obtain the exponential decay result for $\|(\nabla\widetilde{u},\nabla\widetilde{\theta})\|_{L^2}$,
\ben\label{H1-decay}
\|\nabla\widetilde{u}(t)\|_{L^2}+\|\nabla\widetilde{\theta}(t)\|_{L^2}\leq  (\|\nabla{u}_0\|_{L^2}+\|\nabla{\theta}_0\|_{L^2})e^{-C_1 t}.
\een
Combining the estimates \eqref{L2-decay} and \eqref{H1-decay}, we obtain the desired decay result. This completes
the proof of Theorem \ref{main2}. 
\end{proof}

\vskip .2in 
\noindent {\bf Acknowledgments}.
Dong is partially supported by the National Natural Science Foundation of China (No.11871346),
the Natural Science Foundation of Guangdong Province (No. 2018A030313024), the Natural Science Foundation of Shenzhen City (No. JCYJ 20180305125554234) and Research Fund of Shenzhen University (No. 2017056).
Wu was partially supported by NSF under grant DMS 1624146 and the AT\&T Foundation at Oklahoma State University. 
 X. Xu was partially supported by the National Natural Science Foundation of China (No. 11771045 and No.11871087).  N. Zhu was partially supported by the National Natural Science Foundation of China (No. 11771043 and No. 11771045).


\end{document}